\documentclass[oneside,english]{amsart}
\usepackage[T1]{fontenc}
\usepackage[latin9]{inputenc}
\usepackage{textcomp}
\usepackage{amsthm}
\usepackage{amstext}
\usepackage{amssymb}
\usepackage{esint}

\makeatletter
\numberwithin{equation}{section}
\numberwithin{figure}{section}
\theoremstyle{plain}
\newtheorem{thm}{Theorem}
  \theoremstyle{definition}
  \newtheorem{defn}[thm]{Definition}
  \theoremstyle{plain}
  \newtheorem{lem}[thm]{Lemma}
  \theoremstyle{plain}
  \newtheorem{prop}[thm]{Proposition}

\makeatother

\usepackage{babel}

\begin{document}

\title{{\large Random interlacements on Galton-Watson Trees}}

\author{Martin Tassy\\
\emph{\small }\\
{\scriptsize ENS cachan, antenne de bretagne}\\
{\scriptsize }\\
{\scriptsize June 2, 2010}}

\address{Martin Tassy, ENS Cachan, Antenne de Bretagne, Avenue Robert Schumann,
35170 Bruz, France}

\address{email: martin.tassy@ens-cachan.org}

\thanks{Martin Tassy wishes to thank Ji\v{r}� \v{C}ern� for scientific guidance
as well as the ETH Zurich, where this research was conducted, for
its hospitality and Professor A.S. Sznitman who gave him this opportunity
to work in Zurich.}
\begin{abstract}
We study the critical parameter $u^{*}$ of random interlacements
percolation (introduced by A.S Sznitman in \cite{Szn09}) on a Galton-Watson
tree conditioned on the non-extinction event. Starting from the previous
work of A. Teixeira in \cite{Tei09b}, we show that, for a given law
of a Galton-Watson tree, the value of this parameter is a.s. constant
and non-trivial. We also characterize this value as the solution of
a certain equation.
\end{abstract}
\maketitle

\section{introduction}

The aim of this note is to the study random interlacements model on
a Galton-Watson tree. We will mainly be interested in the critical
parameter of the model. In particular we want to understand whether
this parameter is non trivial (that is different from $0$ and $\infty$),
whether it is random and how it depends on the law of the Galton-Watson
tree. Our main theorem answers all of these questions and even goes
further by characterizing the critical parameter as the solution of
a certain equation.

The random interlacements model was recently introduced on $\mathbb{Z}^{d}$,
$d\geq3,$ by A.S. Sznitman in \cite{Szn09} and generalised to arbitrary
transient graphs by A. Teixeira in \cite{Tei09b}. It is a special
dependent site-percolation model where the set $\mathcal{I}$ of occupied
vertices on a transient graph $\left(G,\mathcal{E}\right)$ is constructed
as the trace left on $G$ by a Poisson point process on the space
of double infinite trajectories modulo time shift. The density of
the set $\mathcal{I}$ is determined by a parameter $u>0$ which comes
as a multiplicative parameter of the intensity measure of the Poisson
point process. In this paper, we will not need the complete construction
of the random interlacements percolation. For our purposes it will
be sufficient to know that the law $Q_{u}^{G}$ of the vacant set
$\mathcal{V}=G\setminus\mathcal{I}$ of the random interlacements
at level $u$ is characterized by \begin{equation}
Q_{u}^{G}\left[K\subset\mathcal{V}\right]=e^{-u\mbox{cap}_{G}\left(K\right)},K\subset G\mbox{ finite},\label{eq:formulecap}\end{equation}
where $\mbox{cap}_{G}\left(K\right)$ is the capacity of $K$ in $G$
(see Section 2 for definition). In addition to this formula we will
need the description of the distribution of the vacant cluster containing
a given vertex in the case when $G$ is a tree given in \cite{Tei09b}
which we state in Theorem \ref{thm:augusto perc-1} below. 

The critical parameter $u_{G}^{*}$ of random interlacements on $G$
is defined as:\begin{equation}
u_{G}^{*}=\inf_{u\in\mathbb{R}^{+}}\left\{ u\,:\, Q_{u}^{G}\mbox{-a.s. all connected components of }\mathcal{V}\mbox{ are finite}\right\} .\label{eq:blabla}\end{equation}

In this article we take the graph $G$ to be a Galton-Watson rooted
tree $T$ defined on a probability space $\left(\Omega,\mathcal{A},\mathbb{P}\right)$,
conditioned on non-extinction. We denote by $\varnothing$ the root
of the tree, $\left(\rho_{i}\right)_{i\geq0}$ the offspring distribution
of the non-conditioned Galton-Watson process, $f$ its generating
function and $q$ is the probability of the extinction event. $\bar{\mathbb{P}}$
stands for the conditional law of the Galton-Watson tree on the non-extinction
event and $\bar{\mathbb{E}}$ for the corresponding conditional expectation.
We assume that $T$ is supercritical that is \begin{equation}
\sum_{i=0}^{\infty}i\rho_{i}>1\tag{\ensuremath{A0}}.\label{eq:a0}\end{equation}
If \eqref{eq:a0} is satisfied, then $\bar{\mathbb{P}}$ is well defined
and $T$ is $\bar{\mathbb{P}}$-a.s. transient, see for example Proposition
3.5 and Corollary 5.10 of \cite{LyonsPer}. Thus the random interlacements
are defined on $T$ and in particular $Q_{u}^{T}$ is $\bar{\mathbb{P}}$-a.s.
well-defined.

To state our principal theorem we introduce $\widetilde{T}$ the sub-tree
of $T$ composed of the vertices which have infinite descendence.
We recall the Harris decomposition (cf. Proposition 5.23 in \cite{LyonsPer}):
under $\bar{\mathbb{P}}$, $\widetilde{T}$ is a Galton-Watson tree
with generating function $\widetilde{f}$ given by\begin{equation}
\widetilde{f}\left(s\right)=\frac{f\left(q+\left(1-q\right)s\right)}{1-q}.\label{eq:ftilde}\end{equation}

\begin{thm}
\label{thm:main}Let $T$ be a Galton Watson tree with a law satisfying
\eqref{eq:a0}. Then there exists a non-random constant $u^{*}\in\left(0,\infty\right)$
such that \begin{equation}
\mbox{\ensuremath{u_{T}^{*}=u^{*}}\ensuremath{\,,\,\bar{\mathbb{P}}\mbox{-a.s.}}}\label{eq:b1}\end{equation}
Moreover if we denote by $\mathcal{L}_{\chi}\left(u\right)=\bar{\mathbb{E}}\left[e^{-u\mbox{cap}_{T}\left(\left\{ \varnothing\right\} \right)}\right]$
the annealed probability that the root is vacant at level $u$ (cf
\ref{eq:formulecap}), then $u^{*}$ is the only solution on $\left(0,\infty\right)$
of the equation\begin{equation}
\mbox{ \ensuremath{\left(\widetilde{f}^{-1}\right)'\left(\mathcal{L}_{\chi}\left(u\right)\right)=1}.}\label{eq:b2}\end{equation}

\end{thm}
The main difficulty in proving this theorem is the dependence present
in the model. More precisely, for any $x\in T$, the probability that
$x\in\mathcal{V}$ is given by\[
Q_{u}^{T}\left[x\in\mathcal{V}\right]=e^{-u\mbox{cap}_{T}\left(x\right)}.\]
The capacity $\mbox{cap}_{T}\left(x\right)$ depends on the whole
tree $T$. It is thus not possible to construct the component of the
vacant set containing the root of a given tree as a sequence of independent
generations, which is a key property when proving an analogous statement
for Bernoulli percolation on a Galton-Watson tree.

However, it turns out that despite this dependence, it is possible
to construct a recurrence relation under the annealed measure $\bar{\mathbb{P}}\otimes Q_{u}^{T}$
for a well chosen quantity related to the size of the cluster at a
given point (see \eqref{eq:reccant}). This is done in section 4.
Using this recurrence relation, it is possible to find an annealed
critical parameter $u^{*}$ and show that it is non-trivial. In section
3 we prove that $u_{T}^{*}$ is $\bar{\mathbb{P}}-$a.s. constant
and thus that $u_{T}^{*}=u^{*}$ for $\bar{\mathbb{P}}-$a.e. tree.
Section 2 introduces some preliminary definitions and recall some
useful results for random interlacements on trees.

\section{Definitions and preliminary results}

Let us introduce some notations first. For a given rooted tree $T$
and a vertex $x\in T\setminus\left\{ \varnothing\right\} $, we write
$\hat{x}$ for the closest ancestor of $x$ in $T$, $\left|x\right|$
for its distance to the root, $Z_{x}$ for the number of children
of $x$ and $\mbox{deg}_{T}\left(x\right)$ for its degree. We denote
by $T_{x}$ the sub-tree of $T$ containing $x$ and all its descendants.
If $T$ is a tree with root $\varnothing$, for every child $x$ of
$\varnothing$ we will say that $T_{x}$ is a descendant tree of $T$.
For any infinite rooted tree $T$, we denote by $\widetilde{Z}_{x}$
the number of children of $x$ in $\widetilde{T}$. 

For a tree $T$ and a vertex $x$ of $T$, we denote by $P_{x}^{T}$
the law of a simple discrete time random walk $\left(X_{n}\right)_{n\geq0}$
started at $x$. For every set $K$ in $T$, we use $\widetilde{H}_{K}$
to denote the hitting time of the set $K$ defined by\begin{equation}
\widetilde{H}_{K}=\inf_{n\geq1}\left\{ n\,:\, X_{n}\in K\right\} .\end{equation}
We write $e_{K}^{T}$ for the equilibrium measure of $K$ in $T$
and $\mbox{cap}_{T}\left(K\right)$ for its total mass, also called
capacity of the set $K$:\begin{equation}
e_{K}^{T}\left(x\right)=\mbox{deg}_{T}\left(x\right)P_{x}^{T}\left[\widetilde{H}_{K}=\infty\right]1_{x\in K},\end{equation}
\begin{equation}
\mbox{cap}_{T}\left(K\right)=\sum_{x\in T}e_{K}^{T}\left(x\right)=\sum_{x\in K}\mbox{deg}_{T}\left(x\right)P_{x}^{T}\left[\widetilde{H}_{K}=\infty\right].\label{eq:capequa}\end{equation}

Then, according to Corollary 3.2 of \cite{Tei09b}, the definition
\eqref{eq:blabla} of the critical parameter $u_{T}^{*}$ is equivalent
to \begin{equation}
u_{T}^{*}=\inf_{u\in\mathbb{R}^{+}}\left\{ u:Q_{u}^{T}\left[|\mathcal{C}_{\varnothing}|=\infty\right]>0\right\} .\end{equation}

Finally we recall Theorem 5.1 of \cite{Tei09b} which identifies the
law of the vacant cluster $\mathcal{C}_{x}$ on a fixes tree $T$
with the law of the vacant set left by inhomogeneous Bernoulli site
percolation. This theorem also allows us to compare random interlacements
on $T$ and random interlacements on its descendant trees.
\begin{thm}[Theorem 5.1 of \cite{Tei09b}]
\label{thm:augusto perc-1}Let $T$ be a transient rooted tree with
locally bounded degree. For every vertex $x\in T$ we consider a function
$h_{\textrm{T}}^{x}:T_{x}\rightarrow\left[0,1\right]$ given by:\begin{equation}
h_{\textrm{T}}^{x}\left(z\right)=\mbox{deg}_{T}\left(z\right)P_{z}^{T_{x}}\left[\widetilde{H}_{\left\{ z,\hat{z}\right\} }=\infty\right]P_{z}^{T_{x}}\left[\widetilde{H}_{\hat{z}}=\infty\right]1_{z\neq x}\label{eq:formulef}\end{equation}
and $h_{\textrm{T}}^{x}\left(x\right)=0.$ Conditionally on $\left\{ x\in\mathcal{V}\right\} $,
$\mathcal{C}_{x}\cap T_{x}$ has the same law under $Q_{u}^{T}$ as
the open cluster of $x$ in an inhomogeneous Bernoulli site-percolation
on $T_{x}$ where every site $z\in T_{x}$ is open independently with
probability \begin{equation}
p_{u}\left(z\right)=\exp\left(-uh_{T}^{x}\left(z\right)\right).\label{eq:probaaugusto}\end{equation}

\end{thm}

\section{$\bar{\mathbb{P}}$-a.s. constancy of $u^{*}$}

In this section we prove that for a given Galton-Watson tree $T$
satisfying \eqref{eq:a0} the critical parameter $u_{T}^{*}$ is $\bar{\mathbb{P}}$-a.s
constant. We will use Theorem \ref{thm:augusto perc-1} to prove a
zero-one law for the event $\left\{ Q_{u}^{T}\left[|\mathcal{C}_{\varnothing}|=\infty\right]=0\right\} $.
The proof of this zero-one law is based on the following definition
and lemma which we learnt in \cite{LyonsPer}. We present here its
proof for sake of completeness.
\begin{defn}
\label{lem:hereditary}We say that a property $\mathcal{P}$ of a
tree is inherited if the two following conditions are satisfied:\begin{equation}
\mbox{\ensuremath{T\,}has \ensuremath{\mathcal{P}\,}\ensuremath{\Rightarrow}\ all descendents of \ensuremath{T}trees have \ensuremath{\mathcal{P}}.}\label{eq:condition3.2}\end{equation}
\begin{equation}
\mbox{All finite trees have \ensuremath{\mathcal{P}}.}\end{equation}

The zero-one law for Galton-Watson trees associated to such properties
is: \end{defn}
\begin{lem}
If $\mathcal{P}$ is an inherited property then, for every Galton-Watson
tree process $T$ satisfying \eqref{eq:a0}, $\bar{\mathbb{P}}\left[T\mbox{ has }\mathcal{P}\right]\in\left\{ 0,1\right\} $.\end{lem}
\begin{proof}
Let $E$ be the set of trees that have the property $\mathcal{P}$
and $Z_{\varnothing}$ the number of children of the root. Using condition
\eqref{eq:condition3.2} we can write\begin{eqnarray}
\mathbb{P}\left[T\in E\right] & = & \mathbb{E}\left[\mathbb{P}\left[T\in E|Z_{\varnothing}\right]\right]\label{eq:TdansE}\\
 & \leq & \mathbb{E}\left[\mathbb{P}\left[\left\{ \forall x:\left|x\right|=1,T_{x}\in E\right\} |Z_{\varnothing}\right]\right].\nonumber \end{eqnarray}
Conditionally on $Z_{\varnothing}$, the random trees $\left(T_{x}\right)_{x:\left|x\right|=1}$
are independent and have the same law as $T$. Hence the last inequality
is equivalent to \begin{eqnarray}
\mathbb{P}\left[T\in E\right] & \leq & \mathbb{E}\left[\left(\mathbb{P}\left[T\in E\right]\right)^{Z_{\varnothing}}\right]\label{eq:3.7}\\
 & \leq & f\left(\mathbb{P}\left[T\in E\right]\right).\nonumber \end{eqnarray}
By assumption \eqref{eq:a0} we know that $\rho_{i}>0$ for some $i\geq2$.
Therefore $f$ is strictly convex on $\left(0,1\right)$ with $f\left(q\right)=q$
and $f\left(1\right)=1$. Hence from \eqref{eq:3.7} we have $\mathbb{P}\left[T\in E\right]\in\left[0,q\right]\cup\left\{ 1\right\} $.
Since all finite trees have $\mathcal{P}$, and $\mathbb{P}\left[T\mbox{ is finite}\right]=q$,
we can deduce that $\mathbb{P}\left[T\in E\right]\in\left\{ q,1\right\} $
and consequently \begin{equation}
\bar{\mathbb{P}}\left[T\in E\right]=\frac{\mathbb{P}\left[\left\{ T\in E\right\} \cap\left\{ \left|T\right|=\infty\right\} \right]}{\mathbb{P}\left[\left\{ \left|T\right|=\infty\right\} \right]}\in\left\{ 0,1\right\} \end{equation}
which finishes the proof of the lemma.
\end{proof}
$\,$
\begin{proof}[Proof of \eqref{eq:b1}]
 To prove that $u_{T}^{*}$ is $\ensuremath{\bar{\mathbb{P}}\mbox{-a.s}}$
constant we show first that, for a tree $T$ with root $\varnothing$,
the property $\mathcal{P}_{u}$ defined by \begin{equation}
\mbox{\ensuremath{T}\mbox{ has} \ensuremath{\mathcal{P}_{u}}iff \ensuremath{T}is finite or \ensuremath{Q_{u}^{T}\left[\left|\mathcal{C}_{\varnothing}\right|=\infty\right]=0},}\end{equation}
 is inherited. Since every finite tree has $\mathcal{P}_{u}$ by definition,
we just have to prove the statement \begin{equation}
\left\{ \exists x\in T,\left|x\right|=1:T_{x}\mbox{ has not }\mathcal{P}_{u}\right\} \Rightarrow\left\{ T\mbox{ has not }\mathcal{P}_{u}\right\} .\end{equation}
Let $x$ be a child of the root such that $\left|T_{x}\right|=\infty$
and $T_{x}$ has not $\mathcal{P}_{u}$ which can also be written
$Q_{u}^{T_{x}}\left[\left|\mathcal{C}_{x}\right|=\infty\right]>0$.
Since $T_{x}\subset T$, $T$ is not finite. We will show that $Q_{u}^{T}\left[\left|\mathcal{C}_{\varnothing}\right|=\infty\right]>0$.

From the formula \eqref{eq:formulef} it follows that for every $z\in T_{x}\setminus\left\{ x\right\} $,
$h_{T}^{x}\left(z\right)=h_{T_{x}}^{x}\left(z\right)$. Using Theorem
\ref{thm:augusto perc-1} this means that, conditionally on $\left\{ x\in\mathcal{V}\right\} $,
the law of the cluster $\mathcal{C}_{x}\cap T_{x}$ under $Q_{u}^{T}$
and the law of the cluster $\mathcal{C}_{x}$ under $Q_{u}^{T_{x}}$
are both the same. In particular we have \begin{equation}
Q_{u}^{T}\left[|\mathcal{C}_{x}\cap T_{x}|=\infty|x\in\mathcal{V}\right]=Q_{u}^{T_{x}}\left[\left|\mathcal{C}_{x}\right|=\infty|x\in\mathcal{V}\right].\label{eq:3.10}\end{equation}
Moreover, we see from \eqref{eq:formulef} that $h_{T}^{x}\left(z\right)=h_{T}^{\varnothing}\left(z\right)$
for every $z\in T_{x}\setminus\left\{ x\right\} $. Applying Theorem
\ref{thm:augusto perc-1} to the clusters $\mathcal{C}_{\varnothing}$
and $\mathcal{C}_{x}$, we see that the law of $\mathcal{C}_{\varnothing}\cap T_{x}$
under $Q_{u}^{T}\left[.|\varnothing,x\in\mathcal{V}\right]$ is the
same as the law of $\mathcal{C}_{x}\cap T_{x}$ under $Q_{u}^{T}\left[.|x\in\mathcal{V}\right]$.
In particular we have \begin{equation}
Q_{u}^{T}\left[|\mathcal{C}_{\varnothing}\cap T_{x}|=\infty|\varnothing,x\in\mathcal{V}\right]=Q_{u}^{T}\left[|\mathcal{C}_{x}\cap T_{x}|=\infty|x\in\mathcal{V}\right].\label{eq:3.11}\end{equation}
 Since on $\left\{ \varnothing,x\in\mathcal{V}\right\} $, $\varnothing$
and $x$ are in the same open cluster, we can rewrite \eqref{eq:3.10}
and \eqref{eq:3.11} as \begin{eqnarray}
Q_{u}^{T}\left[|\mathcal{C}_{x}\cap T_{x}|=\infty|\varnothing,x\in\mathcal{V}\right] & = & Q_{u}^{T}\left[|\mathcal{C}_{\varnothing}\cap T_{x}|=\infty|\varnothing,x\in\mathcal{V}\right]\\
 & = & Q_{u}^{T}\left[|\mathcal{C}_{x}\cap T_{x}|=\infty|x\in\mathcal{V}\right]\nonumber \\
 & = & Q_{u}^{T_{x}}\left[\left|\mathcal{C}_{x}\right|=\infty|x\in\mathcal{V}\right]\nonumber \\
 & > & 0,\nonumber \end{eqnarray}
by hypothesis. Finally, since \begin{eqnarray}
Q_{u}^{T}\left[\varnothing,x\in\mathcal{V}\right] & = & e^{-u\mbox{cap}_{T}\left(\left\{ \varnothing,x\right\} \right)}>0,\end{eqnarray}
 this yields \begin{eqnarray}
0 & < & Q_{u}^{T}\left[|\mathcal{C}_{\varnothing}\cap T_{x}|=\infty|\varnothing,x\in\mathcal{V}\right]Q_{u}^{T}\left[\varnothing,x\in\mathcal{V}\right]\leq Q_{u}^{T}\left[\left|\mathcal{C}_{\varnothing}^{u}\right|=\infty\right]\end{eqnarray}
 which finishes the proof that $\mathcal{P}_{u}$ is inherited.

We can now apply Theorem \ref{lem:hereditary} and deduce that $\mathbb{\bar{P}}\left(T\mbox{ has }\mathcal{P}_{u}\right)\in\left\{ 0,1\right\} $.
Thus for every $s\in\mathbb{Q}^{+}$, there exists a set $A_{s}\subset\Omega$
such that $\bar{\mathbb{P}}\left[A_{s}\right]=1$ and $1_{Q_{s}^{T}\left[|\mathcal{C}_{\varnothing}|=\infty\right]>0}$
is constant on $A_{s}$. This yields \begin{equation}
T\rightarrow\inf_{s\in\mathbb{Q}^{+}}\left\{ u:Q_{s}^{T}\left[|\mathcal{C}_{\varnothing}|=\infty\right]>0\right\} \end{equation}
is constant on $A=\cap_{s\in\mathbb{Q}^{+}}A_{s}$ with $\bar{\mathbb{P}}\left[A\right]=1$.
But, since $u\mapsto Q_{u}^{T}\left[|\mathcal{C}_{\varnothing}|=\infty\right]$
is decreasing, we also have: \begin{eqnarray}
u_{*}^{T} & = & \inf_{u\in\mathbb{R}^{+}}\left\{ u:Q_{u}^{T}\left[|\mathcal{C}_{\varnothing}|=\infty\right]>0\right\} \\
 & = & \inf_{u\in\mathbb{Q}^{+}}\left\{ u:Q_{u}^{T}\left[|\mathcal{C}_{\varnothing}|=\infty\right]>0\right\} \mbox{.}\nonumber \end{eqnarray}
It follows directly that $u_{*}$ is $\mathbb{\bar{P}}\mbox{-a.s}$
constant.
\end{proof}

\section{Characterization of $u^{*}$}

In this section we will show that $u^{*}$ is non-trivial and can
be obtained as the root of equation \eqref{eq:b2}. In order to make
our calculation more natural we will work with the modified tree $T'$
obtained by attaching an additional vertex $\Delta$ to the root of
$T$. This change is legitimate only if $T$ and $T'$ have the same
critical parameter $u^{*}$, which is equivalent to \begin{equation}
Q_{u}^{T}\left[\left|\mathcal{C}_{\varnothing}\right|=\infty\right]>0\mbox{ iff }Q_{u}^{T'}\left[\left|\mathcal{C}_{\varnothing}\right|=\infty\right]>0.\label{eq:fifif}\end{equation}
 To prove \eqref{eq:fifif} observe that by Theorem \ref{thm:augusto perc-1}
that \begin{equation}
Q_{u}^{T}\left[\left|\mathcal{C}_{\varnothing}\right|=\infty|\varnothing\in\mathcal{V}\right]=Q_{u}^{T'}\left[\left|\mathcal{C}_{\varnothing}\right|=\infty|\varnothing\in\mathcal{V}\right].\end{equation}
Since $Q_{u}^{T}\left[\varnothing\in\mathcal{V}\right]>0$ and $Q_{u}^{T'}\left[\varnothing\in\mathcal{V}\right]>0$
this is equivalent to\begin{equation}
\frac{Q_{u}^{T}\left[\left|\mathcal{C}_{\varnothing}\right|=\infty\right]}{Q_{u}^{T}\left[\varnothing\in\mathcal{V}\right]}=\frac{Q_{u}^{T'}\left[\left|\mathcal{C}_{\varnothing}\right|=\infty\right]}{Q_{u}^{T'}\left[\varnothing\in\mathcal{V}\right]}\end{equation}
so that \eqref{eq:fifif} holds and $u_{T}^{*}=u_{T'}^{*}$.

If $\left|x\right|\geq1$, $\left(T_{x}\right)'$ is isomorphic to
the tree obtained by attaching $\hat{x}$ to $T_{x}$ . We will thus
identify both trees and write $T_{x}^{'}$ for the tree $T_{x}\cup\left\{ \hat{x}\right\} $.

For every tree $T$ we define the random variable \begin{eqnarray}
\gamma\left(T\right) & = & P_{\varnothing}^{T'}\left[\widetilde{H}_{\Delta}=\infty\right].\label{eq:defgamma}\end{eqnarray}
The random variables $\left(\gamma\left(T_{x}\right)\right)_{\left|x\right|=1}$
are related to the random variable $\chi\left(T\right):=\mbox{cap}_{T'}\left(\varnothing\right)$
by \begin{eqnarray}
\chi\left(T\right) & := & \mbox{cap}_{T}\left(\varnothing\right)\label{eq:relation}\\
 & = & \mbox{cap}_{T'}\left(\varnothing\right)=\sum_{x\in\widetilde{T}:\left|x\right|=1}\gamma\left(T_{x}\right).\end{eqnarray}
The second equality is an easy consequence of definition of the capacity
and the third equality follows, using Markov property, from the following
computation: \begin{eqnarray}
\chi\left(T\right) & = & \left(Z_{\varnothing}+1\right)P_{\varnothing}^{T'}\left[\widetilde{H}_{\varnothing}=\infty\right]\nonumber \\
 & = & \left(Z_{\varnothing}+1\right)\sum_{x\in\widetilde{T}:\left|x\right|=1}P_{x}^{T'}\left[\widetilde{H}_{\varnothing}=\infty\right]P_{\varnothing}^{T'}\left[X_{1}=x\right]\nonumber \\
 & = & \left(Z_{\varnothing}+1\right)\sum_{x\in\widetilde{T}:\left|x\right|=1}\frac{1}{\left(Z_{\varnothing}+1\right)}P_{x}^{T}\left[\widetilde{H}_{\varnothing}=\infty\right]\label{eq:calabelpas}\\
 & = & \sum_{x\in\widetilde{T}:\left|x\right|=1}P_{x}^{T_{x}^{'}}\left[\widetilde{H}_{\varnothing}=\infty\right]=\sum_{x\in\widetilde{T}:\left|x\right|=1}\gamma\left(T_{x}\right).\nonumber \end{eqnarray}
The recursive structure of Galton-Watson tree implies that the random
variables $\left(\gamma\left(T_{x}\right)\right)_{\left|x\right|=1}$
are i.i.d. We can use this property and formula \eqref{eq:ftilde}
to express relation \eqref{eq:relation} in terms of Laplace transforms.
This yields\begin{eqnarray}
\mathcal{L}_{\chi}\left(u\right) & := & \bar{\mathbb{E}}\left[\exp\left(-u\mbox{cap}_{T'}\left(\varnothing\right)\right)\right]\label{eq:laplequa}\\
 & = & \bar{\mathbb{E}}\left[\bar{\mathbb{E}}\bigg[\exp(-u\sum_{x\in\widetilde{T}:\left|x\right|=1}\gamma\left(T_{x}\right))|\widetilde{Z}_{\varnothing}\bigg]\right]\nonumber \\
 & = & \bar{\mathbb{E}}\left[\bar{\mathbb{E}}\bigg[\prod_{x\in\widetilde{T}:\left|x\right|=1}\exp\left(-u\gamma\left(T_{x}\right)\right)|\widetilde{Z}_{\varnothing}\bigg]\right]\nonumber \\
 & = & \bar{\mathbb{E}}\left[\bar{\mathbb{E}}\left[\exp\left(-u\gamma\left(T\right)\right)\right]^{\widetilde{Z}_{\varnothing}}\right]\overset{\eqref{eq:ftilde}}{=}\widetilde{f}\left(\mathcal{L}_{\gamma}\left(u\right)\right).\nonumber \end{eqnarray}
Since $\left(\widetilde{f}^{-1}\right)'=\frac{1}{\widetilde{f}'\circ\widetilde{f}^{-1}}$,
this allows us to write \eqref{eq:b2} as\begin{equation}
\ensuremath{\frac{1}{\widetilde{f}'\circ\widetilde{f}^{-1}\left(\widetilde{f}\left(\mathcal{L}_{\gamma}\left(u\right)\right)\right)}=1}\end{equation}
and thus \begin{equation}
\tag*{(1.5)'}\widetilde{f}'\left(\mathcal{L}_{\gamma}\left(u\right)\right)=1.\label{eq:laprime}\end{equation}
Moreover, $\widetilde{f}$ and $\widetilde{f}'$ being bijective,
the uniqueness of the solution is preserved. Thus from now on we will
consider \eqref{eq:b2} and \ref{eq:laprime} as equivalent.

We will now explicit a relation verified by the annealed probability
that $\mathcal{C}_{\textrm{\ensuremath{\varnothing}}}$ is infinite.
\begin{prop}
\label{pro:u*<infini}For a Galton-Watson process, $r^{u}=\bar{\mathbb{E}}\left[Q_{u}^{T}\left(\left|\mathcal{C}_{\textrm{\ensuremath{\varnothing}}}\right|=\infty\right)\right]$
is the largest root in $\left[0,1\right]$ of the equation\begin{eqnarray}
\mathcal{L}_{\chi}\left(u\right)-r^{u} & = & \widetilde{f}\left(\mathcal{L}_{\gamma}\left(u\right)-r^{u}\right).\label{eq:laplace}\end{eqnarray}
\end{prop}
\begin{proof}
For every vertex $x\in T$, we introduce the notation \begin{equation}
\mathcal{D}_{x}=\left(\max_{y\in\mathcal{C}_{x}\cap\widetilde{T}_{x}}\left|y\right|\right)-\left|x\right|,\label{eq:reccant}\end{equation}
the relative depth of the cluster containing $x$. Since $\left\{ \mathcal{D}_{\textrm{\ensuremath{\varnothing}}}\geq n+1\right\} \subset\left\{ \varnothing\in\mathcal{V}\right\} $,
for every $n\in\mathbb{N}$, we have \begin{equation}
Q_{u}^{T'}\left[\left\{ \varnothing\in\mathcal{V}\right\} \cap\left\{ \mathcal{D}_{\textrm{\ensuremath{\varnothing}}}<n+1\right\} \right]=Q_{u}^{T'}\left(\varnothing\in\mathcal{V}\right)-Q_{u}^{T'}\left(\mathcal{D}_{\textrm{\ensuremath{\varnothing}}}\geq n+1\right).\label{eq:prob}\end{equation}
 We can also write \begin{align}
Q_{u}^{T'}\left[\left\{ \varnothing\in\mathcal{V}\right\} \cap\left\{ \mathcal{D}_{\textrm{\ensuremath{\varnothing}}}<n+1\right\} \right] & =Q_{u}^{T'}\left[\mathcal{D}_{\textrm{\ensuremath{\varnothing}}}<n+1|\varnothing\in\mathcal{V}\right]Q_{u}^{T'}\left[\varnothing\in\mathcal{V}\right]\nonumber \\
 & =Q_{u}^{T'}\left[\cup_{x\in\widetilde{T}:\left|x\right|=1}\left\{ \mathcal{D}_{\textrm{\ensuremath{x}}}<n\right\} |\varnothing\in\mathcal{V}\right]Q_{u}^{T'}\left[\varnothing\in\mathcal{V}\right].\label{eq:formulemagique}\end{align}
According to Theorem \ref{thm:augusto perc-1}, we know that conditionally
on $\left\{ \varnothing\in\mathcal{V}\right\} $, under $Q_{u}^{T'}$,
$\mathcal{C}_{\varnothing}$ has the same law as a cluster obtained
by Bernoulli site-percolation. Hence on $\left\{ \varnothing\in\mathcal{V}\right\} $,
the random variables $\left(\mathcal{D}_{x}\right)_{x\in\widetilde{T}:\left|x\right|=1}$
are independent under $Q_{u}^{T}$. Moreover for all vertices $x\in\widetilde{T}$
such that $\left|x\right|=1$ and every $z\in T_{x}^{'}\setminus\left\{ \varnothing\right\} $
we have the equality $h_{T}^{x}\left(z\right)=h_{T_{x}^{'}}^{x}\left(z\right)$.
This implies that \begin{eqnarray}
Q_{u}^{T'}\left[\mathcal{D}_{x}\geq n|\varnothing\in\mathcal{V}\right] & = & Q_{u}^{T_{x}^{'}}\left[\mathcal{D}_{x}\geq n|\varnothing\in\mathcal{V}\right].\end{eqnarray}
So that we can rewrite \eqref{eq:formulemagique} as\begin{align}
Q_{u}^{T'}\left[\left\{ \varnothing\in\mathcal{V}\right\} \cap\left\{ \mathcal{D}_{\textrm{\ensuremath{\varnothing}}}<n+1\right\} \right] & =\left(\prod_{x\in\widetilde{T}:\left|x\right|=1}Q_{u}^{T'}\left[\mathcal{D}_{x}<n|\varnothing\in\mathcal{V}\right]\right)Q_{u}^{T'}\left[\varnothing\in\mathcal{V}\right]\nonumber \\
 & =\left(\prod_{x\in\widetilde{T}:\left|x\right|=1}Q_{u}^{T_{x}^{'}}\left[\mathcal{D}_{x}<n|\varnothing\in\mathcal{V}\right]\right)Q_{u}^{T'}\left[\varnothing\in\mathcal{V}\right]\label{eq:premierpas}\\
 & =\left(\prod_{x\in\widetilde{T}:\left|x\right|=1}\left(1-\frac{Q_{u}^{T_{x}^{'}}\left[\mathcal{D}_{x}\geq n\right]}{Q_{u}^{T_{x}^{'}}\left[\varnothing\in\mathcal{V}\right]}\right)\right)Q_{u}^{T'}\left[\varnothing\in\mathcal{V}\right].\nonumber \end{align}
According to \eqref{eq:capequa} and \eqref{eq:blabla}, we have\begin{eqnarray}
Q_{u}^{T_{x}^{'}}\left[\varnothing\in\mathcal{V}\right] & = & \exp\left(-u\,\mbox{deg}_{T_{x}^{'}}\left(\varnothing\right)P_{x}^{T_{x}^{'}}\left[\widetilde{H}_{\varnothing}=\infty\right]\right)\label{eq:djadja}\\
 & = & \exp\left(-u\, P_{x}^{T'}\left[\widetilde{H}_{\varnothing}=\infty\right]\right)\nonumber \\
 & = & \exp\left(-u\,\gamma\left(T_{x}\right)\right).\nonumber \end{eqnarray}
Using \eqref{eq:relation} and \eqref{eq:djadja}, \eqref{eq:premierpas}
can be rewritten as \begin{align}
Q_{u}^{T'} & \left[\left\{ \varnothing\in\mathcal{V}\right\} \cap\left\{ \mathcal{D}_{\textrm{\ensuremath{\varnothing}}}<n+1\right\} \right]\nonumber \\
\,\, & =\left(\prod_{x\in\widetilde{T}:\left|x\right|=1}\left(1-Q_{u}^{T_{x}^{'}}\left[\mathcal{D}_{x}\geq n\right]e^{u\gamma\left(T_{x}\right)}\right)\right)e^{-u\sum_{x\in\widetilde{T}:\left|x\right|=1}\gamma\left(T_{x}\right)}\label{eq:lalere}\\
\,\, & =\prod_{x\in\widetilde{T}:\left|x\right|=1}\left(e^{-u\gamma\left(T_{x}\right)}-Q_{u}^{T_{x}^{'}}\left[\mathcal{D}_{x}\geq n\right]\right).\nonumber \end{align}
We can now finally prove \eqref{eq:laplace}. If we denote $r_{n}^{u}=\bar{\mathbb{E}}\left[Q_{u}^{T'}\left(\mathcal{D}_{\textrm{\ensuremath{\varnothing}}}\geq n\right)\right]$,
we have

\begin{align}
\mathcal{L}_{\chi}\left(u\right)-r_{n+1}^{u}:= & \,\,\bar{\mathbb{E}}\left[e^{-u\mbox{cap}_{T'}\left(\varnothing\right)}-Q_{u}^{T'}\left(\mathcal{D}_{\textrm{\ensuremath{\varnothing}}}\geq n+1\right)\right]\nonumber \\
= & \,\,\bar{\mathbb{E}}\left[Q_{u}^{T'}\left(\varnothing\in\mathcal{V}\right)-Q_{u}^{T'}\left(\mathcal{D}_{\textrm{\ensuremath{\varnothing}}}\geq n+1\right)\right]\nonumber \\
\overset{\eqref{eq:prob}}{\,\,\,\,\,\,\,=} & \,\,\bar{\mathbb{E}}\left[Q_{u}^{T'}\left[\left\{ \varnothing\in\mathcal{V}\right\} \cap\left\{ \mathcal{D}_{\textrm{\ensuremath{\varnothing}}}<n+1\right\} \right]\right]\label{eq:recrecfrero}\\
\overset{\eqref{eq:lalere}}{\,\,\,\,\,\,\,=} & \,\,\bar{\mathbb{E}}\left[\prod_{x\in\widetilde{T}:\left|x\right|=1}\left(\exp\left(-u\gamma\left(T_{x}\right)\right)-Q_{u}^{T_{x}}\left(\mbox{ \ensuremath{\mathcal{D}_{x}}}\geq n\right)\right)\right]\nonumber \\
= & \,\,\widetilde{f}\left(\mathcal{L}_{\gamma}\left(u\right)-r_{n}^{u}\right).\nonumber \end{align}
The sequence $\left(r_{n}^{u}\right)_{n\in\mathbb{N}}$ is decreasing
by definition. Therefore it converges and its limit $r^{u}=\bar{\mathbb{E}}\left[Q_{u}^{T\text{\textasciiacute}}\left[\mathcal{D}_{\textrm{\ensuremath{\varnothing}}}=\infty\right]\right]$
verifies \begin{equation}
\begin{array}{ccc}
\mathcal{L}_{\chi}\left(u\right)-r^{u} & = & \widetilde{f}\left(\mathcal{L}_{\gamma}\left(u\right)-r^{u}\right).\end{array}\label{eq:recurrence pn}\end{equation}
The function $\widetilde{f}$ is strictly convex by \eqref{eq:a0}.
Therefore the function \begin{equation}
x\mapsto\mathcal{L}_{\chi}\left(u\right)-\widetilde{f}\left(\mathcal{L}_{\gamma}\left(u\right)-x\right)\end{equation}
 is strictly concave and the equation \begin{equation}
\mathcal{L}_{\chi}\left(u\right)-x=\widetilde{f}\left(\mathcal{L}_{\gamma}\left(u\right)-x\right).\label{eq:lasteq}\end{equation}
has at most two roots in $[0,1$). According to \eqref{eq:laplequa},
$0$ is always a root. Assume that there exists another root $x_{0}$
in $\left(0,1\right)$. Then using the concavity we have \begin{equation}
x>\mathcal{L}_{\chi}\left(u\right)-\widetilde{f}\left(\mathcal{L}_{\gamma}\left(u\right)-x\right)\label{eq:recrecfrade}\end{equation}
 on $\left(0,x_{0}\right)$. The sequence $\left(r_{n}^{u}\right)_{n\in\mathbb{N}}$
is positive and decreasing and verifies \eqref{eq:recrecfrero}. Thus
\eqref{eq:recrecfrade} and an easy recurrence shows that $\forall n\in\mathbb{N}$,
$r_{n}^{u}\geq x_{0}$. Since $r^{u}$ verifies \eqref{eq:recurrence pn},
$r^{u}$ a root of \eqref{eq:lasteq} and $r^{u}$ can only be $x_{0}$
in this case. 

Finally, the tree $T$ have locally finite degree and thus $\left\{ \mathcal{D}_{\textrm{\ensuremath{\varnothing}}}=\infty\right\} =\left\{ \left|\mathcal{C}_{\textrm{\ensuremath{\varnothing}}}\right|=\infty\right\} $.
This yields \begin{equation}
r^{u}=\bar{\mathbb{E}}\left[Q_{u}^{T'}\left[\left|\mathcal{C}_{\textrm{\ensuremath{\varnothing}}}\right|=\infty\right]\right]\end{equation}
 is the largest root of \eqref{eq:laplace} in $\left(0,1\right)$
which finishes the proof of Proposition \ref{pro:u*<infini}.
\end{proof}
We are now able to deduce non-triviality and \ref{eq:laprime} from
the deterministic study of the roots of the equality \eqref{eq:laplace}.
\begin{proof}[Proof of \ref{eq:laprime}]
According to \eqref{eq:laplace}, we have that \begin{equation}
\begin{array}{ccc}
\mathcal{L}_{\chi}\left(u\right)-r^{u} & = & \widetilde{f}\left(\mathcal{L}_{\gamma}\left(u\right)-r^{u}\right).\end{array}\end{equation}
 Using Taylor-Laplace formula, this can be rewritten:\begin{eqnarray}
0 & = & r^{u}-\mathcal{L}_{\chi}\left(u\right)+\widetilde{f}\left(\mathcal{L}_{\gamma}\left(u\right)-r^{u}\right)\nonumber \\
 & \overset{\eqref{eq:laplequa}}{=} & r^{u}-\left(\widetilde{f}\left(\mathcal{L}_{\gamma}\left(u\right)\right)-\widetilde{f}\left(\mathcal{L}_{\gamma}\left(u\right)-r^{u}\right)\right)\label{eq:4.20}\\
 & = & r^{u}\left(1-\widetilde{f}'\left(\mathcal{L}_{\gamma}\left(u\right)\right)-r^{u}\intop_{0}^{1}\left(1-t\right)\widetilde{f}''\left(\mathcal{L}_{\gamma}\left(u\right)-tr^{u}\right)dt\right).\nonumber \end{eqnarray}
From Proposition \ref{pro:u*<infini} and the definition of $u^{*}$,
we can easily deduce that $u^{*}$ is the infimum over $u$ for which
the function \begin{equation}
g_{u}\left(x\right)=1-\widetilde{f}'\left(\mathcal{L}_{\gamma}\left(u\right)\right)-x\intop_{0}^{1}\left(1-t\right)\widetilde{f}''\left(\mathcal{L}_{\gamma}\left(u\right)-tx\right)dt\end{equation}
has a positive root in $\left(0,1\right)$.

We first show that \begin{equation}
\mbox{\ensuremath{g_{u}}has a root in \ensuremath{\left(0,1\right)}iff \ensuremath{g_{u}\left(0\right)>0}.}\label{eq:rootajeje}\end{equation}
 According to \eqref{eq:ftilde}, the function $\widetilde{f}$ is
strictly convex. Therefore, for every $x\in\left(0,1\right)$ we have
\begin{equation}
x\intop_{0}^{1}\left(1-t\right)\widetilde{f}''\left(\mathcal{L}_{\gamma}\left(u\right)-tx\right)dt>0.\end{equation}
 In particular, we have $g_{u}\left(x\right)<g_{u}\left(0\right)$
for every $x\in\left(0,1\right)$. Moreover by definitions of $\gamma$
and $\chi$ we directly have $0<\chi\left(T\right)<\gamma\left(T\right)$
and thus \begin{equation}
\mathcal{L}_{\gamma}\left(u\right)<\mathcal{L}_{\chi}\left(u\right)<1.\end{equation}
Replacing $r^{u}$ by $\mathcal{L}_{\gamma}\left(u\right)$ in the
first line and the third line of \eqref{eq:4.20}, we obtain\begin{equation}
\mathcal{L}_{\gamma}\left(u\right)-\mathcal{L}_{\chi}\left(u\right)+\widetilde{f}\left(0\right)=\mathcal{L}_{\gamma}\left(u\right)g_{u}\left(\mathcal{L}_{\gamma}\left(u\right)\right).\end{equation}
This yields \begin{eqnarray}
g_{u}\left(\mathcal{L}_{\gamma}\left(u\right)\right) & = & \frac{\mathcal{L}_{\gamma}\left(u\right)-\mathcal{L}_{\chi}\left(u\right)}{\mathcal{L}_{\gamma}\left(u\right)}<0.\end{eqnarray}
Since $g_{u}$ is continuous, this finishes the proof of \eqref{eq:rootajeje}.

To finish the proof of \ref{eq:laprime} and non-triviality of $u^{*}$
we shoud show that \ref{eq:laprime} admits exactly one solution $u_{0}\in\left(0,\infty\right)$
and that\begin{equation}
\mbox{\ensuremath{g_{u}\left(0\right)>0}iff \ensuremath{u>u_{0}}.}\label{eq:rootahuhu}\end{equation}
 We have \begin{equation}
g_{u}\left(0\right)=1-\widetilde{f}'\left(\mathcal{L}_{\gamma}\left(u\right)\right)\label{eq:g0}\end{equation}
 with $\widetilde{f}'\left(\cdot\right)=f'\left(q+\left(1-q\right)\cdot\right)$.
Since $f'$ is continuous, we can write\begin{eqnarray}
\lim_{u\rightarrow\infty}\widetilde{f}'\left(\mathcal{L}_{\gamma}\left(u\right)\right) & = & \widetilde{f}'\left(\lim_{u\rightarrow\infty}\mathcal{L}_{\gamma}\left(u\right)\right)=f'\left(q\right)\\
\lim_{u\rightarrow0}\widetilde{f}'\left(\mathcal{L}_{\gamma}\left(u\right)\right) & = & \widetilde{f}'\left(\lim_{u\rightarrow0}\mathcal{L}_{\gamma}\left(u\right)\right)=f'\left(1\right).\nonumber \end{eqnarray}
Moreover since $f\left(q\right)=q$ and $f\left(1\right)=1$, the
strict convexity of $f$ given by \eqref{eq:a0} implies that $f'\left(q\right)<1$
and $f'\left(1\right)>1$. According to \eqref{eq:g0}, this means
that \begin{equation}
\lim_{u\rightarrow0}g_{u}\left(0\right)<0<\lim_{u\rightarrow\infty}g_{u}\left(0\right).\end{equation}
Finally $f'$ being increasing, $\widetilde{f}'$ is also increasing
and $u\mapsto g_{u}\left(0\right)$ is decreasing. Consequently, there
exist a unique $u_{0}\in\left(0,\infty\right)$ such that $g_{u_{0}}\left(0\right)=0$
(which is equivalent to $u_{0}$ is a solution of \ref{eq:laprime}
and we have \eqref{eq:rootahuhu}. This implies that $u^{*}=u_{0}$
and concludes our proof. 
\end{proof}
\bibliographystyle{amsalpha}
\addcontentsline{toc}{section}{\refname}
\bibliography{Martin}

\end{document}